\documentclass{tsp-short}

\usepackage{stmaryrd}
\usepackage{bbm}
\usepackage{url}

\usepackage{color}

\newcommand{\R}{\mathbb{R}}
\newcommand{\N}{\mathbb{N}}

\newcommand{\Z}{\mathbb{Z}}
\newcommand{\mmp}{\mathbb{P}}

\newcommand{\eqdistr}{\stackrel{{\rm d}}{=}}

\newcommand{\todistr}{\stackrel{{\rm d}}{\to}}

\newcommand{\todistrj}{\stackrel{J_1}{\Longrightarrow}}
\newcommand{\todistrm}{\stackrel{M_1}{\Longrightarrow}}

\DeclareMathOperator{\1}{\mathbbm{1}}

\theoremstyle{plain}
\newtheorem{theorem}{Theorem}[section]
\newtheorem{lemma}[theorem]{Lemma}
\newtheorem{corollary}[theorem]{Corollary}

\theoremstyle{definition}

\newtheorem{example}[theorem]{Example}

\theoremstyle{remark}
\newtheorem{remark}[theorem]{Remark}

\begin{document}

\title{Renewal shot noise processes in the case of slowly varying tails}

\author{Zakhar Kabluchko}
\address{Zakhar Kabluchko, Institut f\"ur Mathematische Statistik, Westf\"{a}lische Wilhelms-Uni\-ver\-sit\"{a}t M\"{u}nster, Orl\'eans--Ring 10, 48149 M\"unster, Germany}
\email{{zakhar.kabluchko@uni-muenster.de}}

\author{Alexander Marynych}
\address{Alexander Marynych, Faculty of Cybernetics, Taras Shevchenko National University of Kyiv, 01601 Kyiv, Ukraine}
\email{{marynych@unicyb.kiev.ua}}

\begin{abstract}
We investigate weak convergence of renewal shot noise processes in the case of slowly varying tails of the inter-shot times.
We show that these processes, after an appropriate non-linear scaling, converge in the sense of finite-dimensional distributions to an inverse extremal process.
\end{abstract}



\maketitle

\section{Introduction and main result}
Let $\xi_1,\xi_2,\ldots$ be independent copies of a positive random variable $\xi$. Define the random walk $(S_n)_{n\in\N_0}$ by
$$
S_0:=0,\quad S_n:=\xi_1+\ldots+\xi_n,\quad n\in\N,
$$
and let $(\nu(t))_{t\geq 0}$ be the corresponding first-passage time process for $(S_n)_{n\in\N_0}$:
$$
\nu(t):=\inf\{k\in\N:S_k > t\},\quad t\geq 0.
$$
For a locally bounded and measurable function $h:[0,\infty)\to\R$ define the {\it renewal shot noise} process $(Y(t))_{t\geq 0}$ as follows:
$$
Y(t):=\sum_{k=0}^{\infty}h(t-S_k)\1_{\{S_k\leq t\}}=\sum_{k=0}^{\nu(t)-1}h(t-S_k),\quad t\geq 0.
$$
As is readily seen from the definition, the renewal shot noise processes may serve as models of diverse phenomena of a cumulative nature, in which
an overall effect in some system (say, current in a vacuum tube) is induced by the individual effects of a constantly arriving stream (say, electrons) of particles, claims, customers etc. These processes were used in various fields of applied science and we refer the reader to a recent book \cite{IksanovBook} for a rather complete list of possible applications.

The asymptotic behavior of renewal shot noise processes has attracted a considerable attention in the past years. Their weak convergence
was analyzed in \cite{Iksanov:2013,Iksanov+Kabluchko+Marynych:2016,Iksanov+Marynych+Meiners:2014,Iksanov+Marynych+Meiners:2016-1,Iksanov+Marynych+Meiners:2016-2}, see also \cite{Kluppelberg+Mikosch:1995,Kluppelberg+Mikosch+Scharf:2003,Lane:1984} for the case of Poisson shot noise, and to date there is more or less complete picture of their limiting behavior at least in the sense of convergence of finite-dimensional distributions. The book \cite{IksanovBook} provides a comprehensive survey of the topic.

In this paper we focus on a special class of renewal shot noise processes that has not been studied before, namely, we are interested in the case when the distribution of the  inter-shot times $\xi_1,\xi_2,\ldots$ has a slowly varying tail:
\begin{equation}\label{xi_slow_variation}
\mmp\{\xi>x\}\sim \frac{1}{L(x)},\quad x\to+\infty
\end{equation}
for some $L$ slowly varying at infinity. Without loss of generality we assume that the function $L$ in \eqref{xi_slow_variation} is chosen as strictly increasing, continuous on $[0,\infty)$ with $L(0)=0$, see Lemma \ref{L_increasing_continuous} below. Under these assumptions $L$ possesses the unique inverse function $L^{\leftarrow}$ such that $L^{\leftarrow}(L(t))=L(L^{\leftarrow}(t))=t$ for all $t\geq 0$.

Before stating our main result let us briefly recall some known facts about the behavior of random walks with slowly varying tails of the steps. From the classical result by Darling \cite{Darling:1952}, it is known that every linear normalization $a_n S_n+b_n$ leads to a degenerate limit and in order to obtain a proper limiting behavior one has to consider a non-linear scaling. More precisely, the following convergence in distribution holds
\begin{equation}\label{one_dimensional_rw}
n^{-1}L(S_n)\todistr \mathcal{X},\quad n\to\infty,
\end{equation}
where $\mathcal{X}$ has a standard Fr\'{e}chet distribution, i.e. $\mmp\{\mathcal{X}\leq x\}=e^{-1/x}$, $x>0$. The functional analogue of \eqref{one_dimensional_rw} was established
in \cite{Kasahara:1986} and reads
\begin{equation}\label{flt_rw}
n^{-1}L(S_{[n\cdot ]})\todistrj m(\cdot),\quad n\to\infty,
\end{equation}
on the Skorohod space $D[0,\infty)$. Here and hereafter the notation $\todistrj$ ($\todistrm$) is used to denote weak convergence of random elements on $D[0,\infty)$ endowed with the $J_1$-topology (the $M_1$-topology)\footnote{For the definition of the Skorohod space $D[0,\infty)$, as well as different topologies on it, we refer to the treatise \cite{Whitt:2002}.}. The stochastic process $(m(u))_{u\geq 0}$ is the extremal process of a Poisson process $\mathcal{P}$ on $[0,\infty)\times (0,\infty]$ with intensity $dt \times y^{-2} dy$. That is, if $\mathcal{P}:=\sum_{k}\delta_{(t_k,y_k)}$ where $\delta_x$ denotes the Dirac measure at point $x\in[0,\infty)\times (0,\infty]$, then
$$
m(u):=\max_{k: t_k\leq u}\,y_k,\quad u\geq 0.
$$
Let $(m^{\leftarrow}(u))_{u\geq 0}$ be a generalized inverse of $(m(u))_{u\geq 0}$ defined by
$$
m^{\leftarrow}(u):=\inf\{y\geq 0:m(y)>u\},\quad u\geq 0.
$$
With this notation at hand we can formulate our main result as follows.
\begin{theorem}\label{thm:main}
Let $h:[0,\infty)\to[0,\infty)$ be a locally bounded measurable function such that $h\circ L^{\leftarrow}$ is regularly varying at infinity with index $\alpha\in\R$. If \eqref{xi_slow_variation} holds, then for every $n\in\N$ and $0<u_1<\ldots<u_n<\infty$, we have
\begin{equation}\label{eq:main_theorem}
\left(\frac{Y(L^{\leftarrow}(tu_1))}{th(L^{\leftarrow}(t))},\ldots,\frac{Y(L^{\leftarrow}(tu_n))}{th(L^{\leftarrow}(t))}\right)\todistr (u_1^{\alpha}m^{\leftarrow}(u_1),\ldots,u_n^{\alpha}m^{\leftarrow}(u_n)),\quad t\to\infty.
\end{equation}
\end{theorem}

The properties of the extremal process $(m(u))_{u\geq 0}$ and its inverse $(m^{\leftarrow}(u))_{u\geq 0}$ are well-understood, see \cite{Dwass:1964}, \cite{Resnick:1974} and \cite{Resnick:2008}. In particular, the marginals of $m^{\leftarrow}$ are exponential
\begin{equation}\label{eq:marginals}
\mmp\{m^{\leftarrow}(u)>v\}=e^{-v/u},\quad u,v\geq 0.
\end{equation}
and $(m^{\leftarrow}(u))_{u\geq 0}$ has independent (but not stationary) increments. The two-dimensional distributions of $(m^{\leftarrow}(u))_{u\geq 0}$ as well as the distribution of its increments are also known explicitly, see Propositions 2.2 and 2.3 in \cite{Meerschaer+Scheffler:2005}. Finally, let us note that both $(m(u))_{u\geq 0}$ and its inverse are nondecreasing for $u\geq 0$ and almost surely continuous at every fixed $u\geq 0$, see Proposition 4.7 in \cite{Resnick:2008}.

The rest of the paper is organized as follows. In Section \ref{sec:renewal_theory} we present some relevant results from the renewal theory for random walks with slowly varying tails. Even though, they follow more or less directly from the basic convergence \eqref{flt_rw}, we have not been able to locate them in the literature in the desired generality. The main result is proved in Section \ref{sec:proof}. An auxiliary lemma is given in the Appendix.

\section{Renewal theory for random walks with slowly varying tails}\label{sec:renewal_theory}
Using the well-known continuity of the first-passage time mapping in the $M_1$-topology, see \cite{Whitt:1971}, we obtain from \eqref{flt_rw} the following functional limit theorem for the process $(\nu(t))_{t\geq 0}$.
\begin{theorem}\label{flt_rp}
Under the assumption \eqref{xi_slow_variation} we have
\begin{equation*}
t^{-1}\nu(L^{\leftarrow}(t\cdot ))\todistrm m^{\leftarrow}(\cdot),\quad t\to\infty,
\end{equation*}
on $D[0,\infty)$.
\end{theorem}
\begin{remark}
The convergence of finite-dimensional distributions has been proved by direct calculations in \cite{Meerschaer+Scheffler:2005}, see Theorem 2.1 therein. As has already been mentioned, the process $(m^{\leftarrow}(u))_{u\geq 0}$ is continuous in probability and, for every $t>0$, the process $u\mapsto t^{-1}\nu(L^{\leftarrow}(tu))$ is almost surely nondecreasing, hence it is tempting to apply Theorem 3 in \cite{Bingham:1972} to deduce a strengthened version of Theorem \ref{flt_rp} in the $J_1$-topology. However, as has been noted in \cite{Yamazato:2009}, Bingham's assertion is incorrect. Indeed, as we show next, the $M_1$-convergence in Theorem \ref{flt_rp} cannot be replaced by the $J_1$-convergence, providing another counterexample to Bingham's Theorem 3. To show directly that there is no convergence in the $J_1$-topology, consider the
functional
$$
J(f(\cdot)):=\sup_{u\in[0,1]}|f(u)-f(u-)|,\quad f\in D[0,\infty).
$$
It is $J_1$-continuous on the set of all $f\in D[0,\infty)$ having no jump at $1$ (see Theorem 4.5.5 in \cite{Whitt:2002}), which is a zero set with respect to the law of  $(m^{\leftarrow}(u))_{u\geq 0}$. 
If there were convergence in the $J_1$-topology, the continuous mapping theorem would imply
$$
J(t^{-1}\nu(L^{\leftarrow}(t\cdot )))\todistr J(m^{\leftarrow}(\cdot)),\quad t\to\infty,
$$
but $J(t^{-1}\nu(L^{\leftarrow}(t\cdot )))\leq 1/t\to 0$, as $t\to\infty$, while $\mmp\{J(m^{\leftarrow}(\cdot))>\varepsilon\}>0$, for every $\varepsilon>0$, yielding the desired contradiction. Roughly speaking, the reason for the failure of the $J_1$-convergence is that the jumps of the limit process $(m^{\leftarrow}(u))_{u\geq 0}$ appear as limits of large numbers of small jumps of size $1/t$ in the pre-limit process $u\mapsto t^{-1}\nu(L^{\leftarrow}(tu))$.
\end{remark}

Using that $(m^{\leftarrow}(u))_{u\geq 0}$ is almost surely continuous at $u=1$ and \eqref{eq:marginals}, we recover the one-dimensional result, originally due to Teicher \cite{Teicher:1979} (see also Proposition 2.1 in \cite{Iksanov+Moehle:2008} for an alternative proof by the method of moments).
\begin{corollary}\label{con_rp}
Under the assumption \eqref{xi_slow_variation} we have
$$
\frac {\nu(t)}{L(t)}\todistr \mathcal{E}_1,\quad t\to\infty,
$$
where $\mathcal{E}_1$ has standard exponential distribution.
\end{corollary}

Other important quantities in the renewal theory are {\it the last-value} function $u\mapsto S_{\nu(u)-1}$ and {\it the overshoot} $u\mapsto S_{\nu(u)}-u$. Applying Theorem 13.6.4 in \cite{Whitt:2002} we obtain the following result: for every fixed $u>0$,
\begin{equation}\label{eq:last_value_and_overshot_conv}
\left(\frac{L(S_{\nu(L^{\leftarrow}(tu))-1})}{t},\frac{L(S_{\nu(L^{\leftarrow}(tu))})}{t}-u\right)\todistr \left(m(m^{\leftarrow}(u)-),m(m^{\leftarrow}(u))-u\right),
\end{equation}
as $t\to\infty$. Replacing $t$ in the relation above by $L(t)/u$ yields
\begin{equation}\label{eq:self_similarity}
\left(m(m^{\leftarrow}(u)-),m(m^{\leftarrow}(u))\right)\eqdistr u\left(m(m^{\leftarrow}(1)-),m(m^{\leftarrow}(1))\right),
\end{equation}
as well as
$$
\left(\frac{L(S_{\nu(t)-1})}{L(t)},\frac{L(S_{\nu(t)})}{L(t)}\right)\todistr \left(m(m^{\leftarrow}(1)-),m(m^{\leftarrow}(1))\right),\quad t\to\infty.
$$
Observe that $m(m^{\leftarrow}(1)-)<1$ is the last value of $m$ before reaching the level $1$, whereas $m(m^{\leftarrow}(1))>1$ is the first value of $m$ after reaching the level $1$. In order to calculate the distribution of the limit, note that for all $0<x_1<1$ and $x_2>1$,
\begin{eqnarray*}
&&\mmp\{m(m^{\leftarrow}(1)-)\leq x_1,m(m^{\leftarrow}(1))\leq x_2\}\\
&&\hspace{1cm}=\int_{[0,\,\infty)}\mmp\{m(s-)\leq x_1,m(s)\leq x_2,m^{\leftarrow}(1)\in{\rm d}s\}\\
&&\hspace{1cm}=\int_{[0,\,\infty)}\mmp\{\mathcal{P}([0,s)\times (x_1,+\infty))=0,\mathcal{P}([s,\,s+{\rm d}s]\times (1,x_2])\geq1\}\\
&&\hspace{1cm}=\int_0^{\infty}\exp\left\{-\int_0^s\int_{x_1}^{\infty}\frac{{\rm d}x}{x^2}{\rm d}t\right\}\int_1^{x_2}\frac{{\rm d}x}{x^2}{\rm d}s=x_1\left(1-\frac{1}{x_2}\right).
\end{eqnarray*}
The same formula follows immediately from the fact that the range of $(m(u))_{u\geq 0}$ form a Poisson point process on $(0,\infty)$ with intensity $x^{-1} {\rm d} x$, see Theorem~1 and Corollary~1 in~\cite{Resnick:1974} (by the way, the same is true for the range of the inverse process $(m^{\leftarrow}(u))_{u\geq 0}$).
Summarizing, we obtain the following result.
\begin{theorem}\label{thm:last_value_and_overshot_conv}
As $t\to\infty$,
$$
\left(\frac{L(S_{\nu(t)-1})}{L(t)},\frac{L(S_{\nu(t)})}{L(t)}\right)\todistr (\mathcal{U},\mathcal{V}),
$$
where $\mathcal{U}$ and $\mathcal{V}$ are independent, $\mathcal{U}$ has uniform distribution on $(0,1)$ and $\mathcal{V}$ has distribution $\mmp\{\mathcal{V}>x\}=\frac{1}{x}$ for $x\geq 1$.
\end{theorem}

\begin{example}
Consider a simple symmetric random walk on $\Z^2$ starting at the origin. Denote the times at which the random walk returns to the origin by $\xi_1+\ldots+\xi_k$, $k\in\N$,  so that $\xi_k$ is the length of the $k$-th excursion away from the origin. Then, $\xi_1,\xi_2,\ldots$ are independent identically distributed random variables and it is known~\cite{dvoretzky_erdoes:1951} that
$\mmp\{\xi_1>n\} \sim \pi/\log n$ as $n\to\infty$. Denoting by $\nu(u)$ the number of visits to the origin up to time $u>0$, we obtain from Theorem~\ref{flt_rp} that
$$
t^{-1} \nu(\exp (t \cdot )) \todistrm \pi^{-1}m^{\leftarrow}(\cdot),\quad t\to\infty,
$$
on $D[0,\infty)$, a result which was obtained in~\cite{Kasahara:1986extremal}. Recently, an independent proof has been given in~\cite{nandori_shen:2016}.
\end{example}

\section{Proof of Theorem \ref{thm:main}}\label{sec:proof}
We start with an auxiliary lemma.
\begin{lemma}\label{uniformity_in_rv}
Let $L:[0,\infty)\to[0,\infty)$ be a strictly increasing continuous function slowly varying at $+\infty$ such that $L(0)=0$, $\lim_{x\to+\infty}L(x)=+\infty$ and let $h:[0,\infty)\to[0,\infty)$ be a function such that
$h\circ L^{\leftarrow}$ is regularly varying at $+\infty$ with index $\alpha\in\R$. Then
\begin{itemize}
\item $h$ is slowly varying at $+\infty$;
\item for every $u>0$ and $\varepsilon\in (0,u)$, it holds
\begin{equation}\label{eq:uniformity_in_lemma}
\lim_{t\to+\infty}\sup_{y\in [0,\,u-\varepsilon]}\left|\frac{h(L^{\leftarrow}(tu)-L^{\leftarrow}(ty))}{h(L^{\leftarrow}(t))}-u^{\alpha}\right|=0.
\end{equation}
\end{itemize}
\end{lemma}
\begin{proof}
Slow variation of $h$ follows immediately from the regular variation of $h\circ L^{\leftarrow}$, slow variation of $L$, and the equality
$$
\frac{h(ut)}{h(t)}=\frac{h(L^{\leftarrow}(L(tu)))}{h(L^{\leftarrow}(L(t)))},\quad u>0.
$$
To prove \eqref{eq:uniformity_in_lemma} we argue as follows. Since $L$ is slowly varying and increasing, we have, for $0<a<b$,
\begin{equation}\label{eq:l_inverse_fastly_vaying}
\lim_{t\to+\infty}\frac{L^{\leftarrow}(ta)}{L^{\leftarrow}(tb)}=0.
\end{equation}
From \eqref{eq:l_inverse_fastly_vaying} we immediately obtain,
\begin{equation}\label{eq:l_uniform1}
\lim_{t\to+\infty}\sup_{y\in[0,\,u-\varepsilon]}\left|\frac{L^{\leftarrow}(tu)-L^{\leftarrow}(ty)}{L^{\leftarrow}(tu)}-1\right|=0,\quad u>0,\quad \varepsilon\in(0,u)
\end{equation}
By the triangle inequality,
\begin{eqnarray*}
\left|\frac{h(L^{\leftarrow}(tu)-L^{\leftarrow}(ty))}{h(L^{\leftarrow}(t))}-u^{\alpha}\right|&\leq&\frac{h(L^{\leftarrow}(tu)-L^{\leftarrow}(ty))}{h(L^{\leftarrow}(tu))}\left|\frac{h(L^{\leftarrow}(tu))}{h(L^{\leftarrow}(t))}-u^{\alpha}\right|\\
&+&u^{\alpha}\left|\frac{h(L^{\leftarrow}(tu)-L^{\leftarrow}(ty))}{h(L^{\leftarrow}(tu))}-1\right|.
\end{eqnarray*}
It remains to show that
\begin{equation}\label{eq:L_uniform2}
\lim_{t\to\infty}\sup_{y\in [0,u-\varepsilon]}\left|\frac{h(L^{\leftarrow}(tu)-L^{\leftarrow}(ty))}{h(L^{\leftarrow}(tu))}-1\right|=0
\end{equation}
By the uniform convergence theorem for slowly varying functions, see Theorem 1.2.1 in \cite{Bingham+Goldie+Teugels:1987}, for every $\delta>0$ there exists $s_0>0$ such that,
\begin{equation}\label{eq:L_uniform3}
1-\delta\leq \frac{h(\lambda s)}{h(s)}\leq 1+\delta,
\end{equation}
for all $s\geq s_0$ and all $\lambda\in [1/2,2]$. From \eqref{eq:l_uniform1} it follows that there exists $t_1>0$ such that 
$$
\frac{1}{2}\leq \frac{L^{\leftarrow}(tu)-L^{\leftarrow}(ty)}{L^{\leftarrow}(tu)} \leq 2\quad\text{and}\quad L^{\leftarrow}(tu)\geq s_0,
$$
for all $t\geq t_1$ and $y\in [0,u-\varepsilon]$. Combining pieces together, we see that \eqref{eq:L_uniform2} follows from \eqref{eq:L_uniform3} with $s:= L^{\leftarrow}(tu)$ and $\lambda:=\frac{L^{\leftarrow}(tu)-L^{\leftarrow}(ty)}{L^{\leftarrow}(tu)}$. The proof of the lemma is complete.
\end{proof}

\begin{proof}[Proof of Theorem \ref{thm:main}] We use the following representation for the renewal shot noise process:
$$
Y(u):=\int_{[0,\,u]}h(u-y){\rm d}\nu(y),\quad u\geq 0.
$$
Using this formula we infer, for $t>0$,
\begin{eqnarray*}
\frac{Y(L^{\leftarrow}(ut))}{th(L^{\leftarrow}(t))}&=&\int_{[0,\,L^{\leftarrow}(ut)]}\frac{h(L^{\leftarrow}(ut)-y)}{th(L^{\leftarrow}(t))}{\rm d}\nu(y),\quad u\geq 0.
\end{eqnarray*}
Recall that we assume strict monotonicity and continuity of $L$ and also that $L(0)=0$. Change of the variable formula gives
\begin{eqnarray*}
\frac{Y(L^{\leftarrow}(ut))}{th(L^{\leftarrow}(t))}&=&\int_{[0,\,u]}\frac{h(L^{\leftarrow}(ut)-L^{\leftarrow}(zt))}{h(L^{\leftarrow}(t))}{\rm d_z}\frac{\nu(L^{\leftarrow}(zt))}{t},
\end{eqnarray*}
where ${\rm d_z}$ denotes the differential with respect to $z$. Fix some $n\in\N$, $0 < u_1<\ldots <u_n$ and $\gamma_1,\ldots,\gamma_n\geq 0$. Fix also $\varepsilon\in(0,u_1)$. According to the Cram\'{e}r-Wold device, it is enough to check that
\begin{equation}\label{eq:cramer_wold}
\sum_{i=1}^{n}\gamma_i\frac{Y(L^{\leftarrow}(u_it))}{th(L^{\leftarrow}(t))}\todistr \sum_{i=1}^{n}\gamma_i u_i^{\alpha}m^{\leftarrow}(u_i),\quad t\to\infty.
\end{equation}
Rewrite the left-hand side of \eqref{eq:cramer_wold} as follows:
\begin{eqnarray*}
&&\hspace{-1.5cm}\sum_{i=1}^{n}\gamma_i\frac{Y(L^{\leftarrow}(u_it))}{th(L^{\leftarrow}(t))}\\
&=&\int_{[0,\infty)}\left(\sum_{i=1}^{n}\gamma_i\frac{h(L^{\leftarrow}(u_it)-L^{\leftarrow}(zt))}{h(L^{\leftarrow}(t))}\1_{\{0\leq z\leq u_i-\varepsilon\}}\right){\rm d_z}\frac{\nu(L^{\leftarrow}(zt))}{t}\\
&+&\int_{[0,\infty)}\left(\sum_{i=1}^{n}\gamma_i\frac{h(L^{\leftarrow}(u_it)-L^{\leftarrow}(zt))}{h(L^{\leftarrow}(t))}\1_{\{u_i-\varepsilon <z\leq u_i\}}\right){\rm d_z}\frac{\nu(L^{\leftarrow}(zt))}{t}\\
&=&Z_{1,\varepsilon}(t)+Z_{2,\varepsilon}(t).
\end{eqnarray*}
According to formula \eqref{eq:uniformity_in_lemma} in Lemma \ref{uniformity_in_rv},
$$
Z_{1,\varepsilon}(t)=\int_{[0,\infty)}\left(\sum_{i=1}^{n}\gamma_i(u_i^{\alpha}+o(1))\1_{\{0\leq z\leq u_i-\varepsilon\}}\right){\rm d_z}\frac{\nu(L^{\leftarrow}(zt))}{t},
$$
where the term $o(1)$ does not depend on $z\in [0, u_i-\varepsilon]$ and tends to zero as $t\to\infty$, whence
$$
Z_{1,\varepsilon}(t)=\sum_{i=1}^{n}\gamma_i(u_i^{\alpha}+o(1))\frac{\nu(L^{\leftarrow}((u_i-\varepsilon)t))}{t}.
$$
From Theorem \ref{flt_rp} we obtain
$$
Z_{1,\varepsilon}(t)\todistr\sum_{i=1}^{n}\gamma_i u_i^{\alpha}m^{\leftarrow}(u_i-\varepsilon)=:Z_{1,\varepsilon}.
$$
Obviously, $m^{\leftarrow}(u_i-\varepsilon)\uparrow m^{\leftarrow}(u_i-)$ almost surely, as $\varepsilon\to 0+$ for all $i=1,\ldots,n$. Therefore, almost surely,
$$
Z_{1,\varepsilon}\to \sum_{i=1}^{n}\gamma_i u_i^{\alpha}m^{\leftarrow}(u_i-)=\sum_{i=1}^{n}\gamma_i u_i^{\alpha}m^{\leftarrow}(u_i),\quad \varepsilon\to 0+,
$$
where the second equality follows from the almost sure continuity of $(m^{\leftarrow}(u))_{u\geq 0}$ at arbitrary fixed $u\geq 0$, see Proposition 4.7 in \cite{Resnick:2008}\footnote{Since $(m^{\leftarrow}(u))_{u\geq 0}$ is nondecreasing, stochastic continuity at fixed $u\geq 0$ implies almost sure continuity at $u$.}. According to Theorem 4.2 in \cite{billingsley_book} it remains to show
\begin{equation}\label{eq:billingsley}
\lim_{\varepsilon\to 0+}\limsup_{t\to\infty}\mmp\{Z_{2,\varepsilon}(t)>\delta\}=0,
\end{equation}
for every fixed $\delta>0$. Note that \eqref{eq:billingsley} follows once we can show that
\begin{equation*}
\lim_{\varepsilon\to 0+}\limsup_{t\to\infty}\mmp\{\nu(L^{\leftarrow}(ut))-\nu(L^{\leftarrow}((u-\varepsilon)t))>0\}=0,
\end{equation*}
for every fixed $u>0$. We have
\begin{eqnarray*}
\mmp\{\nu(L^{\leftarrow}(ut))-\nu(L^{\leftarrow}((u-\varepsilon)t))>0\}&=&\mmp\{\nu(L^{\leftarrow}(ut))-\nu(L^{\leftarrow}((u-\varepsilon)t))\geq 1\}\\
&=&\mmp\{S_{\nu(L^{\leftarrow}((u-\varepsilon)t))}\leq L^{\leftarrow}(ut)\}\\
&=&\mmp\{L(S_{\nu(L^{\leftarrow}((u-\varepsilon)t))})/t\leq u\}\\
&\to&\mmp\{m(m^{\leftarrow}(u-\varepsilon))\leq u\},\quad t\to\infty
\end{eqnarray*}
by formula \eqref{eq:last_value_and_overshot_conv}. Finally,
$$
\mmp\{m(m^{\leftarrow}(u-\varepsilon))\leq u\}\overset{\eqref{eq:self_similarity}}{=}\mmp\{m(m^{\leftarrow}(1))\leq u/(u-\varepsilon)\}=\frac{\varepsilon}{u},
$$
since $m(m^{\leftarrow}(1))$ has a Pareto distribution, see Theorem \ref{thm:last_value_and_overshot_conv}. The last expression tends to zero, as $\varepsilon\to 0+$. This completes the proof of Theorem \ref{thm:main}.
\end{proof}

\section{Appendix}
The lemma below shows that without loss of generality the slowly varying function $L$ in \eqref{xi_slow_variation} can be chosen continuous and {\it strictly} increasing on $[0,\infty)$. Analogous statements regarding the existence of a {\it nondecreasing} asymptotically equivalent function are well-known, see p.~17 in \cite{Seneta:1976} or Proposition 1.3.4 and Corollary 1.3.5 in \cite{Bingham+Goldie+Teugels:1987}.
\begin{lemma}\label{L_increasing_continuous}
Let $L:[0,\infty)\to [0,\infty)$ be a nondecreasing function slowly varying at $+\infty$ with $\lim_{x\to\infty}L(x)=+\infty$. Then there exists a strictly increasing and continuous function $L_1:[0,\infty)\to [0,\infty)$ such that $L_1(0)=0$ and $L(x)\sim L_1(x)$, $x\to\infty$.
\end{lemma}
\begin{proof}
Using the representation theorem for slowly varying functions \cite[Theorem 1.3.1]{Bingham+Goldie+Teugels:1987} together with Corollary 1.3.5 in the same reference we have
$$
L(x)=c(x)\exp\left\{\int_{0}^x \varepsilon(u)u^{-1}{\rm d}u\right\},\quad x\geq 0,
$$
for some measurable function $c(\cdot)$ such that $c(x)\to c\in(0,\infty)$, as $x\to\infty$, and a nonnegative function $\varepsilon(\cdot)$ such that $\varepsilon(x)\to 0$, as $x\to\infty$. Define the function $\varepsilon_1$ as follows:
$$
\varepsilon_1(u):=\begin{cases}
        u, &   \text{if} \ u\leq 1,   \\
        \varepsilon(u), & \text{if} \ u>1\text{ and } \varepsilon(u)>0,\\
				1/u, & \text{if} \ u>1\text{ and } \varepsilon(u)=0.
\end{cases}
$$
Clearly, $b:=\int_0^{\infty}(\varepsilon(u)-\varepsilon_1(u))u^{-1}{\rm d}u$ is a finite constant. Set
$$
L_1(x):=ce^b\left(\exp\left\{\int_{0}^x \varepsilon_1(u)u^{-1}{\rm d}u\right\}-1\right),\quad x\geq 0.
$$
Since $\varepsilon_1$ is positive, $L_1$ is strictly increasing. The continuity of $L_1$, the relation $L_1(x)\sim L(x)$, $x\to\infty$, and the equality $L_1(0)=0$ are trivial. The proof is complete.
\end{proof}

\subsection*{Acknowledgements}
The work of A.~Marynych was supported by the Alexander von Humboldt Foundation.

\bibliographystyle{plain}
\bibliography{bibliography}

\begin{thebibliography}{10}

\bibitem{billingsley_book}
P.~Billingsley.
\newblock {\em Convergence of probability measures.}
\newblock Chichester: Wiley, 2nd edition, 1999.

\bibitem{Bingham:1972}
N.~H. Bingham.
\newblock Limit theorems for occupation times of {M}arkov processes.
\newblock {\em Z. Wahrscheinlichkeitstheor. Verw. Geb.}, 17:1--22, 1971.

\bibitem{Bingham+Goldie+Teugels:1987}
N.~H. Bingham, C.~M. Goldie, and J.~L. Teugels.
\newblock {\em Regular variation}, volume~27 of {\em Encyclopedia of
  Mathematics and its Applications}.
\newblock Cambridge University Press, Cambridge, 1987.

\bibitem{Darling:1952}
D.~A. Darling.
\newblock The influence of the maximum term in the addition of independent
  random variables.
\newblock {\em Trans. Amer. Math. Soc.}, 73:95--107, 1952.

\bibitem{dvoretzky_erdoes:1951}
A.~{Dvoretzky} and P.~{Erd\H{o}s}.
\newblock {Some problems on random walk in space.}
\newblock {Proc. Berkeley Sympos. math. Statist. Probability, California, 1951,
  353--367.}, 1951.

\bibitem{Dwass:1964}
M.~Dwass.
\newblock Extremal processes.
\newblock {\em Ann. Math. Statist}, 35:1718--1725, 1964.

\bibitem{Iksanov:2013}
A.~Iksanov.
\newblock Functional limit theorems for renewal shot noise processes with
  increasing response functions.
\newblock {\em Stoch. Process. Appl.}, 123(6):1987--2010, 2013.

\bibitem{IksanovBook}
A.~Iksanov.
\newblock {\em Perturbed random walks, perpetuities and random processes with
  immigration}.
\newblock Birkh\"{a}user, 2016+.

\bibitem{Iksanov+Kabluchko+Marynych:2016}
A.~Iksanov, Z.~Kabluchko, and A.~Marynych.
\newblock Weak convergence of renewal shot noise processes in the case of
  slowly varying normalization.
\newblock {\em Stat. Probab. Lett.}, 114:67--77, 2016.

\bibitem{Iksanov+Marynych+Meiners:2014}
A.~Iksanov, A.~Marynych, and M.~Meiners.
\newblock Limit theorems for renewal shot noise processes with eventually
  decreasing response functions.
\newblock {\em Stoch. Process. Appl.}, 124(6):2132--2170, 2014.

\bibitem{Iksanov+Marynych+Meiners:2016-1}
A.~Iksanov, A.~Marynych, and M.~Meiners.
\newblock Asymptotics of random processes with immigration {I}: scaling limits.
\newblock {\em Bernoulli}, 2016.
\newblock In press.

\bibitem{Iksanov+Marynych+Meiners:2016-2}
A.~Iksanov, A.~Marynych, and M.~Meiners.
\newblock Asymptotics of random processes with immigration {II}: convergence to
  stationarity.
\newblock {\em Bernoulli}, 2016.
\newblock In press.

\bibitem{Iksanov+Moehle:2008}
A.~Iksanov and M.~M{\"o}hle.
\newblock On the number of jumps of random walks with a barrier.
\newblock {\em Adv. Appl. Probab.}, 40(1):206--228, 2008.

\bibitem{Kasahara:1986extremal}
Y.~{Kasahara}.
\newblock Extremal process as a substitution for ``one-sided stable process
  with index $0$''.
\newblock {Stochastic processes and their applications, Proc. Int. Conf.,
  Nagoya, Japan 1985, Lect. Notes Math. 1203, 90--100.}, 1986.

\bibitem{Kasahara:1986}
Y.~Kasahara.
\newblock A limit theorem for sums of i.i.d.\ random variables with slowly
  varying tail probability.
\newblock {\em J. Math. Kyoto Univ.}, 26(3):437--443, 1986.

\bibitem{Kluppelberg+Mikosch:1995}
C.~Kl{\"u}ppelberg and T.~Mikosch.
\newblock Explosive {P}oisson shot noise processes with applications to risk
  reserves.
\newblock {\em Bernoulli}, 1(1-2):125--147, 1995.

\bibitem{Kluppelberg+Mikosch+Scharf:2003}
C.~Kl{\"u}ppelberg, T.~Mikosch, and A.~Sch{\"a}rf.
\newblock Regular variation in the mean and stable limits for {P}oisson shot
  noise.
\newblock {\em Bernoulli}, 9(3):467--496, 2003.

\bibitem{Lane:1984}
J.~Lane.
\newblock The central limit theorem for the {P}oisson shot-noise process.
\newblock {\em J. Appl. Probab.}, 21(2):287--301, 1984.

\bibitem{Meerschaer+Scheffler:2005}
M.~M. Meerschaert and H.-P. Scheffler.
\newblock Limit theorems for continuous time random walks with slowly varying
  waiting times.
\newblock {\em Stat. Probab. Lett.}, 71(1):15--22, 2005.

\bibitem{nandori_shen:2016}
P.~N\'{a}ndori and Z.~Shen.
\newblock Logarithmic scaling of planar random walk's local times.
\newblock {\em Preprint, arXiv: 1603.07587}, 2016.

\bibitem{Resnick:1974}
S.~I. Resnick.
\newblock Inverses of extremal processes.
\newblock {\em Adv. Appl. Probab.}, 6:392--406, 1974.

\bibitem{Resnick:2008}
S.~I. Resnick.
\newblock {\em Extreme values, regular variation and point processes}.
\newblock Springer Series in Operations Research and Financial Engineering.
  Springer, New York, 2008.
\newblock Reprint of the 1987 original.

\bibitem{Seneta:1976}
E.~Seneta.
\newblock {\em Regularly varying functions}.
\newblock Lecture Notes in Mathematics, Vol. 508. Springer-Verlag, Berlin-New
  York, 1976.

\bibitem{Teicher:1979}
H.~Teicher.
\newblock Rapidly growing random walks and an associated stopping time.
\newblock {\em Ann. Probab.}, 7(6):1078--1081, 1979.

\bibitem{Whitt:1971}
W.~Whitt.
\newblock Weak convergence of first passage time processes.
\newblock {\em J. Appl. Probab.}, 8:417--422, 1971.

\bibitem{Whitt:2002}
W.~Whitt.
\newblock {\em Stochastic-process limits. An introduction to stochastic-process
  limits and their application to queues}.
\newblock Springer Series in Operations Research. Springer-Verlag, New York,
  2002.

\bibitem{Yamazato:2009}
M.~Yamazato.
\newblock On a ${J}_1$-convergence theorem for stochastic processes on
  ${D}[0,\infty)$ having monotone sample paths and its applications ({T}he 8th
  workshop on stochastic numerics).
\newblock {\em RIMS K\^{o}ky\^{u}roku}, 1620:109--118, 2009.

\end{thebibliography}
\end{document}